\documentclass[a4paper,12pt]{article}
\usepackage{graphicx}
\usepackage{float}
\usepackage{latexsym}
\usepackage{amssymb}
\usepackage{amsmath}
\usepackage{amsthm}
\usepackage{cite}
\usepackage{bm}
\usepackage{epstopdf}

\newtheorem{theorem}{Theorem}[section]

\newtheorem{example}[theorem]{Example}
\newtheorem{proposition}[theorem]{Proposition}

\newtheorem{definition}[theorem]{Definition}

\newcommand{\R}{\mathbb{R}}

\begin{document}

\title{Helicoidal surfaces of non-lightlike frontals in Lorentz-Minkowski 3-space}
\author{Kaixin Yao\textsuperscript{1}, Wei Zhang\textsuperscript{2,3*}\\
{\small \it \textsuperscript{1}School of Science, Yanshan University, Qinhuangdao 066004, P. R. China}\\
{\small \it \textsuperscript{1}e-mail: yaokx@ysu.edu.cn}\\
{\small \it \textsuperscript{2}School of Mathematics and Statistics, Yili Normal University, Yining 835000, P. R. China}\\
{\small \it \textsuperscript{3}Institute of Applied Mathematics, Yili Normal University, Yining 835000, P. R. China}\\
{\small \it \textsuperscript{2,3}e-mail: zhangw257@nenu.edu.cn}
  }

\date{\today}

\maketitle
\begin{abstract}
In this paper, we define two types of helicoidal surfaces of non-lightlike frontals in Lorentz-Minkowski 3-space and investigate when they become lightcone framed base surfaces. Moreover, by constructing appropriate diffeomorphic transformations and using the criteria of $(i,j)$-cusps and $(i,j)$-cuspidal edges, we establish identification theorems for the singular types of both 1-type and 2-type helicoidal surfaces on their singular loci.
\end{abstract}

\renewcommand{\thefootnote}{\fnsymbol{footnote}}
\footnote[0]{*Corresponding author.}

\section{Introduction}

Lorentz-Minkowski 3-space \(\mathbb{R}^3_1\) is a fundamental geometric model for describing flat spacetime in general relativity and mathematical physics, endowed with a pseudo inner product of signature \((-,+,+)\) (cf. \cite{Lo,ON}). The geometry and singularity analysis of curves and surfaces in Lorentz-Minkowski 3-space not only possess deep mathematical significance but also find important applications in physics, such as relativity, optics, and wavefront propagation (cf. \cite{Capozziello,HKP}). In particular, the distinction between timelike, spacelike and lightlike vectors corresponds to observers, spatial directions and lightcone structures in physics, making the theory of surfaces in this space both geometrically and physically meaningful.

In relativity, spacetime is regarded as a four-dimensional Lorentzian manifold, whose three-dimensional spatial sections can often be locally described by \(\mathbb{R}^3_1\). Helicoidal surfaces  (or constant angle surfaces), with their helical symmetry, can model the spacetime structure around rotating black holes, spiral wavefront propagation, or physical field distributions with angular momentum (cf. \cite{Gourgoulhon}). Moreover, the concepts of ``frontals'' and ``framed surfaces'' are particularly important in describing wavefront singularities, such as caustics (cf. \cite{FT02,FT01,HT01,LLW2025,PTZ,YLLP}). Therefore, studying helicoidal surfaces and their singularities in Lorentz-Minkowski space not only contributes to understanding the local structure of spacetime geometry but also provides mathematical models for simulating relativistic phenomena (cf. \cite{BMS}).

\begin{figure}[t]
	\begin{center}
		\includegraphics*[width=16cm]{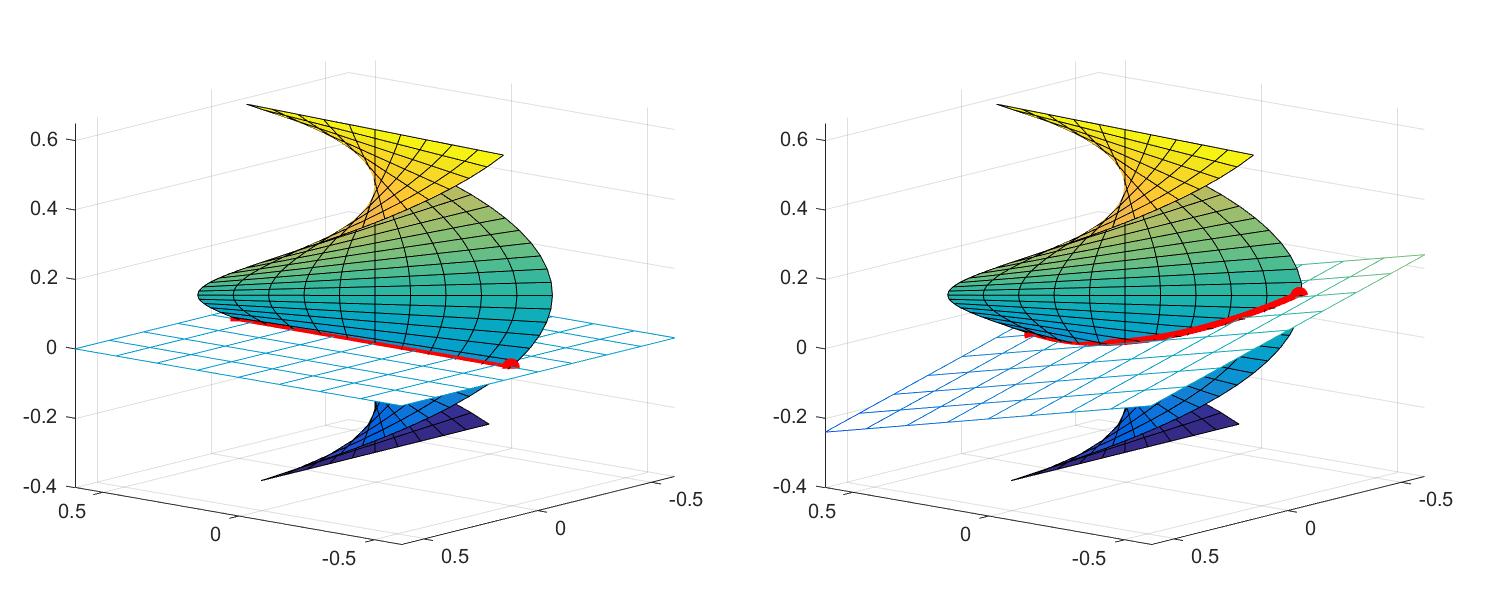}
	\end{center}
	\caption{Left: Intersection between a helicoidal string world sheet
		and
		a plane is a straight line. Right: Same situation, but now the intersection is
		curved (cf. \cite{BMS}).}
	\label{fig:helic}
\end{figure}

The study of helicoidal surfaces originates from the generalization of Bour's theorem in Euclidean geometry. In classical differential geometry, there are a lot of investigations of helicoidal surfaces not only as minimal surfaces but also as constant mean curvature surfaces (cf. \cite{COP2022,GAS01}), and it has attracted considerable attention in singular differential geometry in recent years (cf. \cite{HHM,MS2025,TT}). Moreover, within the framework of Riemannian geometry, recent scholars have studied helicoidal surfaces in different ambient spaces (cf. \cite{COP2024,LO-Y2017,LO-Y2024,MOPP2016,Onnis2017,OPP2019,Pellegrino1,Pellegrino2,WCL2020}). 

It is particularly worth noting the work of Nakatsuyama, Saji, Shimada and Takahashi (cf. \cite{TT}). In \cite{TT}, helicoidal surfaces generated not only by regular curves but also by curves with singular points, for which the use of frontals in the Euclidean plane is particularly useful. The helicoidal surface of a frontal can be naturally considered as a generalised framed base surface. The basic invariants,  curvatures, and the criteria for singularities of the helicoidal surfaces of frontals can be got by using the curvatures of Legendre curves.

The study of singularities of such surfaces in $\R_1^3$ introduces challenges and phenomena absent in the Euclidean case, due primarily to the indefinite metric and the resulting classification of points into spacelike, timelike, and lightlike types. Bridging singularity theory of frontals with geometry in Lorentz-Minkowski 3-space, we consider the $(x,z)$ Lorentz-Minkowski plane into $(x, y,z)$ Lorentz-Minkowski 3-space and give a curve in the $(x,z)$ Lorentz-Minkowski plane, so called the profile curve. We define the 1-type helicoidal surface along the $x$-direction and the 2-type helicoidal surface along the $z$-direction in Lorentz-Minkowski 3-space. On the one hand, if the profile curve has a singular point, then the helicoidal surface automatically has singular points. Even if the profile curve is regular, the helicoidal surface may have singular points. On the other hand, if the profile curve is mixed-type in the $(x,z)$ Lorentz-Minkowski plane, then the helicoidal surface is automatically mixed-type. Even if the profile curve is non-lightlike, the helicoidal surface may be also mixed-type. 

The remained part of this paper is organized as follows. In Section 3, we introduce two types of helicoidal surfaces constructed from non-lightlike frontals and investigate their fundamental differential geometric properties, including singularity conditions and criteria for being spacelike, timelike, or lightlike. Moreover, we show that when $\delta = 1$, these helicoidal surfaces become lightcone framed base surfaces and derive their basic invariants. Section 4 is devoted to the classification of singularities appearing on these helicoidal surfaces. By constructing appropriate diffeomorphisms, we reduce the singularity analysis to that of planar curves and establish necessary and sufficient conditions for the surfaces to admit $(i,j)$-cuspidal edges. Finally, in Section 5, we provide concrete examples to illustrate the theoretical results, accompanied by figures showing the surfaces and their singular loci.

All maps and manifolds considered here are differentiable of class $C^{\infty}$ unless stated otherwise.

\section{Priliminaries}
Let $\mathbb{R}^3$ be the Euclidean 3-space. For any vectors $\bm{x} = (x_1, x_2, x_3), \bm{y} = (y_1, y_2, y_3) \in \mathbb{R}^3,$ their pseudo inner product is defined by
$$\langle \bm{x}, \bm{y} \rangle = -x_1y_1 + x_2y_2 + x_3y_3.$$
$(\mathbb{R}^3, \langle,\rangle)$ is called Lorentz-Minkowski 3-space and denoted by $\mathbb{R}_1^3.$ The pseudo wedge product of $\bm{x} = (x_1, x_2, x_3)$ and $\bm{y} = (y_1, y_2, y_3)$ is
$$\bm{x} \wedge \bm{y} =
\left|
\begin{array}{ccc}
	-\bm{e}_1 & \bm{e}_2 & \bm{e}_3\\
	x_1 & x_2 & x_3\\
	y_1 & y_2 & y_3
\end{array}
\right|,$$
where $\{\bm{e}_1, \bm{e}_2, \bm{e}_3\}$ is the canonical basis of $\mathbb{R}_1^3.$

For any non-zero vector $\bm{x} \in \mathbb{R}_1^3 \setminus \{\bm{0}\},$ it is called spacelike, timelike or lightlike if $\langle \bm{x}, \bm{x} \rangle$ is positive, negative or zero. $\bm{0}$ is a spacelike vector. The set $LC^*$ including all lightlike vectors in $\mathbb{R}_1^3$ is called the ligntcone. $\mathbb{R}_1^2$ is a subspace of $\mathbb{R}_1^3$ with the signature $(-,+).$ Let $\Delta = S_1^1 \cup H^1,$ where $$S_1^1 = \{\bm{x} \in \mathbb{R}_1^2 | \langle \bm{x}, \bm{x} \rangle = 1\}$$ and
$$H^1 = \{\bm{x} \in \mathbb{R}_1^2 | \langle \bm{x}, \bm{x} \rangle = -1\}.$$

\begin{definition}{\rm (\cite{YL})}
	$I$ is an interval of $\mathbb{R}.$ $(\gamma, \bm\nu) : I \rightarrow \mathbb{R}_1^2 \times \Delta$ is called a non-lightlike Legendre curve if $\langle \gamma'(u), \bm\nu(u) \rangle = 0$ for all $u \in I.$ $\gamma : I \rightarrow \mathbb{R}_1^2$ is called a non-lightlike frontal if there exists a smooth map $\bm\nu : I \rightarrow \Delta$ such that $(\gamma, \bm\nu)$ is a non-lightlike Legendre curve. 
\end{definition}

Let $\bm\nu(u) = (a(u), b(u))$ and $\delta = a^2(u) - b^2(u).$ Define $\bm\mu(u) = (b(u), a(u)).$ Then $\bm\mu : I \rightarrow \Delta$ satisfies $\langle \bm\nu(u), \bm\mu(u) \rangle = 0.$ The Frenet-type formula of $(\gamma, \bm\nu)$ is
$$\begin{pmatrix}
	\bm\nu'(u)\\
	\bm\mu'(u)
\end{pmatrix}
=
\begin{pmatrix}
	0 & l(u)\\
	l(u) & 0
\end{pmatrix}
\begin{pmatrix}
	\bm\nu(u)\\
	\bm\mu(u)
\end{pmatrix},
\gamma'(u) = \beta(u)\bm\mu(u),$$
where $l(u) = \delta\langle \bm\nu'(u), \bm\mu(u) \rangle$ and $\beta(u) = \delta\langle \gamma'(t), \bm\mu(u) \rangle.$ The pair $(l,\beta)$ is called the curvature of $(\gamma, \bm\nu).$

\begin{definition}{\rm (\cite{ML})}
	{\rm $U$ is a domain in $\mathbb{R}^2.$ $(\bm{x}, \ell^+, \ell^-) : U \rightarrow \mathbb{R}_1^3 \times \Delta_4$ is called a lightcone framed surface if there exist smooth functions $\alpha, \beta : U \rightarrow \mathbb{R}$ such that $\bm{x}_u(u,v) \wedge \bm{x}_v(u,v) = \alpha(u,v)\ell^+(u,v) + \beta(u,v)\ell^-(u,v)$ for all $(u,v) \in U,$ where $\Delta_4 = \{(\bm{v}, \bm{w}) \in LC^* \times LC^* | \langle \bm{v}, \bm{w} \rangle = -2\}.$ We call $\bm{x} : U \rightarrow \mathbb{R}_1^3$ a lightcone framed base surface if there exists $(\ell^+, \ell^-) : U \rightarrow \Delta_4$ such that $(\bm{x}, \ell^+, \ell^-)$ is a lightcone framed surface.}
\end{definition}

Define $\bm{t}(u,v) = -\dfrac{1}{2}\ell^+(u,v) \wedge \ell^-(u,v)$. Then $\langle \bm{t}(u,v), \bm{t}(u,v) \rangle = 1.$ $\{\ell^+(u,v), \ell^-(u,v), \bm{t}(u,v) \}$ is a moving frame on $\bm{x}.$ We have the following formulas
$$\begin{pmatrix}
	\bm{x}_u(u,v) \\
	\bm{x}_v(u,v)
\end{pmatrix} =
\begin{pmatrix}
	a_1(u,v) & b_1(u,v) & c_1(u,v) \\
	a_2(u,v) & b_2(u,v) & c_2(u,v)
\end{pmatrix}
\begin{pmatrix}
	\ell^+(u,v) \\
	\ell^-(u,v) \\
	\bm{t}(u,v)
\end{pmatrix},$$
$$\begin{pmatrix}
	\ell^+_u(u,v) \\
	\ell^-_u(u,v) \\
	\bm{t}_u(u,v)
\end{pmatrix} =
\begin{pmatrix}
	e_1(u,v) & 0 & 2g_1(u,v) \\
	0 & -e_1(u,v) & 2f_1(u,v) \\
	f_1(u,v) & g_1(u,v) & 0
\end{pmatrix}
\begin{pmatrix}
	\ell^+(u,v) \\
	\ell^-(u,v) \\
	\bm{t}(u,v)
\end{pmatrix},$$
$$\begin{pmatrix}
	\ell^+_v(u,v) \\
	\ell^-_v(u,v) \\
	\bm{t}_v(u,v)
\end{pmatrix} =
\begin{pmatrix}
	e_2(u,v) & 0 & 2g_2(u,v) \\
	0 & -e_2(u,v) & 2f_2(u,v) \\
	f_2(u,v) & g_2(u,v) & 0
\end{pmatrix}
\begin{pmatrix}
	\ell^+(u,v) \\
	\ell^-(u,v) \\
	\bm{t}(u,v)
\end{pmatrix},$$
where
$$\begin{array}{ll}
	a_1(u,v) = -\frac{1}{2}\langle \bm{x}_u(u,v), \ell^-(u,v) \rangle, & a_2(u,v) = -\frac{1}{2}\langle \bm{x}_v(u,v), \ell^-(u,v) \rangle, \\
	b_1(u,v) = -\frac{1}{2}\langle \bm{x}_u(u,v), \ell^+(u,v) \rangle, & b_2(u,v) = -\frac{1}{2}\langle \bm{x}_v(u,v), \ell^+(u,v) \rangle, \\
	c_1(u,v) = \langle \bm{x}_u(u,v), \bm{t}(u,v) \rangle, & c_2(u,v) = \langle \bm{x}_v(u,v), \bm{t}(u,v) \rangle, \\
	e_1(u,v) = -\frac{1}{2}\langle \ell^+_u(u,v), \ell^-(u,v) \rangle, & e_2(u,v) = -\frac{1}{2}\langle \ell^+_v(u,v), \ell^-(u,v) \rangle, \\
	f_1(u,v) = -\frac{1}{2}\langle \bm{t}_u(u,v), \ell^-(u,v) \rangle, & f_2(u,v) = -\frac{1}{2}\langle \bm{t}_v(u,v), \ell^-(u,v) \rangle, \\
	g_1(u,v) = -\frac{1}{2}\langle \bm{t}_u(u,v), \ell^+(u,v) \rangle, & g_2(u,v) = -\frac{1}{2}\langle \bm{t}_v(u,v), \ell^+(u,v) \rangle.
\end{array}$$
The twelve functions are called basic invariants of the lightcone framed surface $(\bm{x}, \ell^+, \ell^-).$

\section{Helicoidal surfaces of non-lightlike frontals}
In this section, we define two types of helicoidal surfaces of non-lightlike frontals and discuss their differential geometric properties.

\begin{definition}
	Let $(\gamma, \bm\nu) : I \rightarrow \mathbb{R}_1^2 \times \Delta$ be a non-lightlike Legendre curve and $\gamma(u) = (x_1(u), x_2(u)).$ We call
	$$\bm{r}_1 : I \times \mathbb{R} \rightarrow \mathbb{R}_1^3, ~ \bm{r}_1(u,v) = (x_1(u) + \lambda v, x_2(u)\sin v, x_2(u)\cos v)$$
	and
	$$\bm{r}_2 : I \times \mathbb{R} \rightarrow \mathbb{R}_1^3, ~ \bm{r}_2(u,v) = (x_1(u)\cosh v, x_1(u)\sinh v, x_2(u) + \lambda v)$$
	a 1-type helicoidal surface and a 2-type helicoidal surface, respectively, where $\lambda$ is a non-zero constant.
\end{definition}

\subsection{1-type helicoidal surfaces}\label{sec}
By calculation, we have
$$\begin{aligned}
	\frac{\partial \bm{r}_1}{\partial u}(u,v) =& \beta(u)(b(u), a(u)\sin v, a(u)\cos v),\\
	\frac{\partial \bm{r}_1}{\partial v}(u,v) =& (\lambda, x_2(u)\cos v, -x_2(u)\sin v),\\
	\frac{\partial \bm{r}_1}{\partial u}(u,v) \wedge \frac{\partial \bm{r}_1}{\partial v}(u,v) =& \beta(u)\big(a(u)x_2(u), b(u)x_2(u)\sin v + \lambda a(u) \cos v,\\
	&~ b(u)x_2(u)\cos v - \lambda a(u)\sin v\big)
\end{aligned}$$
and
$$\left\langle \frac{\partial \bm{r}_1}{\partial u}(u,v) \wedge \frac{\partial \bm{r}_1}{\partial v}(u,v), \frac{\partial \bm{r}_1}{\partial u}(u,v) \wedge \frac{\partial \bm{r}_1}{\partial v}(u,v) \right\rangle = \beta^2(u)(\lambda^2 a^2(u) - \delta x_2^2(u)).$$
So we have following propositions.
\begin{proposition}
	The 1-type helicoidal surface $\bm{r}_1$ is singular at $(u_0, v_0) \in I \times \mathbb{R}$ if and only if $\beta(u_0) = 0$ or $a(u_0) = x_2(u_0) = 0.$
\end{proposition}

\begin{proposition}
	The 1-type helicoidal surface $\bm{r}_1$ is spacelike, timelike or lightlike at $(u_0, v_0) \in I \times \mathbb{R}$ if and only if $\beta(u_0) \neq 0, ~ (a(u_0), x_2(u_0)) \neq (0,0)$ and $\lambda^2a^2(u_0) - \delta x_2^2(u_0)$ is negative, positive or zero.
\end{proposition}

Next, we discuss the sufficient condition that 1-type helicoidal surfaces become lightcone framed base surfaces.

\begin{proposition}
	When $\delta = 1,$ the 1-type helicoidal surface $\bm{r}_1$ is a lightcone framed base surface.
\end{proposition}
\begin{proof}
	For the 1-type helicoidal surface $\bm{r}_1(u,v) = (x_1(u) + \lambda v, x_2(u)\sin v, x_2(u)\cos v),$ let
	$$\begin{aligned}
		\ell_1^+(u,v) =& (a(u), b(u)\sin v + \cos v, b(u)\cos v - \sin v),\\
		\ell_1^-(u,v) =& (a(u), b(u)\sin v - \cos v, b(u)\cos v + \sin v).
	\end{aligned}$$
	Then $$\frac{\partial \bm{r}_1}{\partial u}(u,v) \wedge \frac{\partial \bm{r}_1}{\partial v}(u,v) = \frac{\beta(u)(x_2(u) + \lambda a(u))}{2} \ell_1^+(u,v) + \frac{\beta(u)(x_2(u) - \lambda a(u))}{2} \ell_1^-(u,v).$$
	The 1-type helicoidal surface $\bm{r}_1$ is a lightcone framed base surface.
\end{proof}

Let $$\bm{t}_1(u,v) = -\frac{1}{2}\ell_1^+(u,v) \wedge \ell_1^-(u,v) = (b(u), a(u)\sin v, a(u)\cos v).$$
Then we have the following formulas
$$\begin{pmatrix}
	\dfrac{\partial \bm{r}_1}{\partial u}(u,v)
	\vspace{1.5ex}\\
	\dfrac{\partial \bm{r}_1}{\partial v}(u,v)
\end{pmatrix} =
\begin{pmatrix}
	0 & 0 & \beta(u) \\
	\dfrac{\lambda a(u) + x_2(u)}{2} & \dfrac{\lambda a(u) - x_2(u)}{2} & -\lambda b(u)
\end{pmatrix}
\begin{pmatrix}
	\ell_1^+(u,v)\\
	\ell_1^-(u,v)\\
	\bm{t}_1(u,v)
\end{pmatrix},$$
$$\begin{pmatrix}
	\dfrac{\partial \ell_1^+}{\partial u}(u,v)
	\vspace{1.5ex}\\
	\dfrac{\partial \ell_1^-}{\partial u}(u,v)
	\vspace{1.5ex}\\
	\dfrac{\partial \bm{t}_1}{\partial u}(u,v)
\end{pmatrix} =
\begin{pmatrix}
	0 & 0 & l(u) \\
	0 & 0 & l(u) \\
	\dfrac{l(u)}{2} & \dfrac{l(u)}{2} & 0
\end{pmatrix}
\begin{pmatrix}
	\ell_1^+(u,v)\\
	\ell_1^-(u,v)\\
	\bm{t}_1(u,v)
\end{pmatrix},$$
$$\begin{pmatrix}
	\dfrac{\partial \ell_1^+}{\partial v}(u,v)
	\vspace{1.5ex}\\
	\dfrac{\partial \ell_1^-}{\partial v}(u,v)
	\vspace{1.5ex}\\
	\dfrac{\partial \bm{t}_1}{\partial v}(u,v)
\end{pmatrix} =
\begin{pmatrix}
	b(u) & 0 & -a(u) \\
	0 & -b(u) & a(u) \\
	\dfrac{a(u)}{2} & -\dfrac{a(u)}{2} & 0
\end{pmatrix}
\begin{pmatrix}
	\ell_1^+(u,v)\\
	\ell_1^-(u,v)\\
	\bm{t}_1(u,v)
\end{pmatrix}.$$

\subsection{2-type helicoidal surfaces}
Similar to the Section \ref{sec}, we claim the following conclusions about 2-type helicoidal surfaces without proof.

\begin{proposition}
	The 2-type helicoidal surface $\bm{r}_2$ is singular at $(u_0, v_0) \in I \times \mathbb{R}$ if and only if $\beta(u_0) = 0$ or $b(u_0) = x_1(u_0) = 0.$
\end{proposition}

\begin{proposition}
	The 2-type helicoidal surface $\bm{r}_2$ is spacelike, timelike or lightlike at $(u_0, v_0) \in I \times \mathbb{R}$ if and only if $\beta(u_0) \neq 0, ~ (b(u_0), x_1(u_0)) \neq (0,0)$ and $\lambda^2b^2(u_0) - \delta x_1^2(u_0)$ is negative, positive or zero.
\end{proposition}
\begin{proposition}
	When $\delta = 1,$ the 2-type helicoidal surface $\bm{r}_2$ is a lightcone framed base surface.
\end{proposition}
Let
$$\begin{aligned}
	\ell_2^+(u,v) =& (a(u)\cosh v + \sinh v, a(u)\sinh v + \cosh v, b(u)),\\
	\ell_2^-(u,v) =& (a(u)\cosh v - \sinh v, a(u)\sinh v - \cosh v, b(u))
\end{aligned}$$
and
$$\bm{t}_2(u,v) = -\frac{1}{2}\ell_2^+(u,v) \wedge \ell_2^-(u,v) = (b(u)\cosh v, b(u)\sinh v, a(u)).$$
We have the following formulas
$$\frac{\partial \bm{r}_2}{\partial u}(u,v) \wedge \frac{\partial \bm{r}_2}{\partial v}(u,v) = \frac{\beta(u)(x_1(u) - \lambda b(u))}{2} \ell_2^+(u,v) + \frac{\beta(u)(x_1(u) + \lambda b(u))}{2} \ell_2^-(u,v),$$
$$\begin{pmatrix}
	\dfrac{\partial \bm{r}_2}{\partial u}(u,v)
	\vspace{1.5ex}\\
	\dfrac{\partial \bm{r}_2}{\partial v}(u,v)
\end{pmatrix} =
\begin{pmatrix}
	0 & 0 & \beta(u) \\
	\dfrac{-\lambda b(u) + x_1(u)}{2} & \dfrac{-\lambda b(u) - x_1(u)}{2} & \lambda a(u)
\end{pmatrix}
\begin{pmatrix}
	\ell_2^+(u,v)\\
	\ell_2^-(u,v)\\
	\bm{t}_2(u,v)
\end{pmatrix},$$
$$\begin{pmatrix}
	\dfrac{\partial \ell_2^+}{\partial u}(u,v)
	\vspace{1.5ex}\\
	\dfrac{\partial \ell_2^-}{\partial u}(u,v)
	\vspace{1.5ex}\\
	\dfrac{\partial \bm{t}_2}{\partial u}(u,v)
\end{pmatrix} =
\begin{pmatrix}
	0 & 0 & l(u) \\
	0 & 0 & l(u) \\
	\dfrac{l(u)}{2} & \dfrac{l(u)}{2} & 0
\end{pmatrix}
\begin{pmatrix}
	\ell_2^+(u,v)\\
	\ell_2^-(u,v)\\
	\bm{t}_2(u,v)
\end{pmatrix},$$
$$\begin{pmatrix}
	\dfrac{\partial \ell_2^+}{\partial v}(u,v)
	\vspace{1.5ex}\\
	\dfrac{\partial \ell_2^-}{\partial v}(u,v)
	\vspace{1.5ex}\\
	\dfrac{\partial \bm{t}_2}{\partial v}(u,v)
\end{pmatrix} =
\begin{pmatrix}
	a(u) & 0 & -b(u) \\
	0 & -a(u) & b(u) \\
	\dfrac{b(u)}{2} & -\dfrac{b(u)}{2} & 0
\end{pmatrix}
\begin{pmatrix}
	\ell_2^+(u,v)\\
	\ell_2^-(u,v)\\
	\bm{t}_2(u,v)
\end{pmatrix}.$$

\section{Singularities of helicoidal surfaces}
In this section, we discuss the singularity identification problem of helicoidal surfaces.
\begin{definition}{\rm (\cite{NN})}
	\begin{enumerate}
		\item [{\rm (1)}] Let $\gamma : (I, 0) \rightarrow (\mathbb{R}^2, 0)$ be a curve germ. We say that $\gamma$ is an $(i,j)$-cusp at $0$ if $\gamma$ is $\mathcal A$-equivalent to the germ $u \mapsto (u^i, u^j)$ at the origin, where $(i,j) = (2,3), (2,5), (3,4), (3,5).$
		\item [{\rm (2)}] Let $f : (\mathbb{R}^2, 0) \rightarrow (\mathbb{R}_1^3, 0)$ be a map germ. We say that $f$ is an $(i,j)$-cuspidal edge at $0$ if $f$ is $\mathcal A$-equivalent to the germ $(u,v) \mapsto (u, v^i, v^j)$ at the origin, where $(i,j) = (2,3), (2,5), (3,4), (3,5).$
	\end{enumerate}	
\end{definition}

For $(i,j)$-cusp on a curve $\gamma,$ the following criteria are known (cf. \cite{IRP}).
\begin{proposition}
	Let $\gamma : I \rightarrow \mathbb{R}^2$ be a smooth curve with a singularity $u_0 \in I.$
	\begin{enumerate}
		\item [{\rm (1)}] $\gamma$ has a $(2,3)$-cusp at $u_0$ if and only if $\det(\gamma''(u_0), \gamma'''(u_0)) \neq 0.$
		\item [{\rm (2)}] $\gamma$ has a $(2,5)$-cusp at $u_0$ if and only if $\gamma''(u_0) \neq \bm{0},$ $\gamma'''(u_0) = k \gamma''(u_0)$ for some constant $k$ and $\det(\gamma''(u_0), 3\gamma^{(5)}(u_0) - 10k \gamma^{(4)}(u_0)) \neq 0.$
		\item [{\rm (3)}] $\gamma$ has a $(3,4)$-cusp at $u_0$ if and only if $\gamma''(u_0) = \bm{0}$ and $\det(\gamma'''(u_0), \gamma^{(4)}(u_0)) \neq 0.$
		\item [{\rm (4)}] $\gamma$ has a $(3,5)$-cusp at $u_0$ if and only if $\gamma''(u_0) = \bm{0},~ \det(\gamma'''(u_0), \gamma^{(4)}(u_0)) = 0$ and $\det(\gamma'''(u_0), \gamma^{(5)}(u_0)) \neq 0.$
	\end{enumerate}
\end{proposition}

\subsection{Singularities of 1-type helicoidal surfaces}
Define the following two maps 
$\varphi_1 : I \times \mathbb{R} \rightarrow I \times \mathbb{R}, ~ (u,v) \mapsto (\bar u, \bar v)$ and $\psi_1 : \mathbb{R}_1^3 \rightarrow \mathbb{R}_1^3, ~ (x,y,z) \mapsto (\bar x, \bar y, \bar z)$ by 
$$\varphi_1(u,v) = (u, x_1(u) + \lambda v)$$
and
$$\psi_1(x,y,z) = \left(x, y \sin\frac{x}{\lambda} + z \cos\frac{x}{\lambda}, -y \cos\frac{x}{\lambda} + z \sin\frac{x}{\lambda}\right),$$
respectively. Their Jacobian matrices are 
$$J_{\varphi_1}(u,v) =
\begin{pmatrix}
	1 & 0 \\
	l(u)b(u) & \lambda
\end{pmatrix}$$
and
$$J_{\psi_1}(x,y,z) =
\begin{pmatrix}
	1 & 0 & 0
	\vspace{1.5ex}\\
	\dfrac{\partial \bar y}{\partial x} & \sin\dfrac{x}{\lambda} & \cos\dfrac{x}{\lambda}
	\vspace{1.5ex}\\
	\dfrac{\partial \bar z}{\partial x} & -\cos\dfrac{x}{\lambda} & \sin\dfrac{x}{\lambda}
\end{pmatrix}.$$
So $\varphi_1$ and $\psi_1$ are both invertible, where the inverse map of $\varphi_1$ is $\varphi_1^{-1}(\bar u, \bar v) = \left(\bar u, \dfrac{\bar v - x_1(\bar u)}{\lambda}\right).$

For the 1-type helicoidal surface $\bm{r}_1(u,v) = (x_1(u) + \lambda v, x_2(u)\sin v, x_2(u)\cos v),$ we define a surface
$$\begin{aligned}
	&\psi_1 \circ \bm{r}_1 \circ \varphi_1^{-1}(\bar u, \bar v)\\
	=& \psi_1 \left(\bar v, x_2(\bar u) \sin \frac{\bar v - x_1(\bar u)}{\lambda}, x_2(\bar u) \cos \frac{\bar v - x_1(\bar u)}{\lambda}\right)\\
	=& \left(\bar v, x_2(\bar u) \cos \frac{x_1(\bar u)}{\lambda}, x_2(\bar u) \sin \frac{x_1(\bar u)}{\lambda}\right)
\end{aligned}$$
and a curve
$$\gamma_1(\bar u) = \left(x_2(\bar u) \cos \frac{x_1(\bar u)}{\lambda}, x_2(\bar u) \sin \frac{x_1(\bar u)}{\lambda}\right).$$
Then the 1-type helicoidal surface $\bm{r}_1$ is an $(i,j)$-cuspidal edge at $(u_0, v_0)$ if and only if the surface $\psi_1 \circ \bm{r}_1 \circ \varphi_1^{-1}$ is an $(i,j)$-cuspidal edge at $(u_0, x_1(u_0) + \lambda v_0).$ This is equivalent to the curve $\gamma_1$ being an $(i,j)$-cusp at $\bar u = u_0.$ Note that $\gamma$ and $\gamma_1$ are diffeomorphic except at $x_2(u_0) = 0.$ We only consider $\gamma_1$ at $x_2(u_0) = 0.$ Since the diffeomorphism $\varphi_1$ maps $u$ to $\bar u = u,$ we write $\gamma_1(\bar u)$ as $\gamma_1(u).$

\begin{theorem}\label{th}
	Let $\bm{r}_1$ be the 1-type helicoidal surface and $(u_0, v_0)$ be a singular point of $\bm{r}_1.$ Assume that $x_2(u_0) = 0.$ We have the followings.
	\begin{enumerate}
		\item [{\rm (1)}] If $\beta(u_0) = 0, ~ a(u_0) = 0,$ then for any $v \in \mathbb{R},$ the surface $\bm{r}_1$ is a $(3,5)$-cuspidal edge at $(u_0, v)$ if and only if $\beta'(u_0)l(u_0) \neq 0.$ In this case, $\bm{r}_1$ does not have $(2,3)$-cuspidal edges, $(2,5)$-cuspidal edges or $(3,4)$-cuspidal edges.
		\item [{\rm (2)}] If $\beta(u_0) = 0, ~ a(u_0) \neq 0,$ then for any $v \in \mathbb{R},$ the surface $\bm{r}_1$ is a $(2,5)$-cuspidal edge at $(u_0, v)$ if and only if $\beta'(u_0)l(u_0) \neq 0.$ In this case, $\bm{r}_1$ does not have $(2,3)$-cuspidal edges, $(3,4)$-cuspidal edges or $(3,5)$-cuspidal edges.
		\item [{\rm (3)}] If $\beta(u_0) \neq 0, ~ a(u_0) = 0,$ then for any $v \in \mathbb{R},$ the surface $\bm{r}_1$ is a $(2,3)$-cuspidal edge at $(u_0, v)$ if and only if $l(u_0) \neq 0,$ $\bm{r}_1$ is a $(3,4)$-cuspidal edge at $(u_0, v)$ if and only if $l(u_0) = 0$ and $l'(u_0) \neq 0.$ In this case, $\bm{r}_1$ does not have $(2,5)$-cuspidal edges or $(3,5)$-cuspidal edges.
	\end{enumerate}
\end{theorem}
\begin{proof}
	For the curve $\gamma_1(u) = \left(x_2(u) \cos \dfrac{x_1(u)}{\lambda}, x_2(u) \sin \dfrac{x_1(u)}{\lambda}\right),$ we have
	$$\gamma_1^{(i)}(u) = \left(A_i(u)\cos\frac{x_1(u)}{\lambda} + B_i(u)\sin\frac{x_1(u)}{\lambda}, -B_i(u)\cos\frac{x_1(u)}{\lambda} + A_i(u)\sin\frac{x_1(u)}{\lambda}\right)$$
	and
	$$\begin{aligned}
		&\det(\gamma_1^{(i)}(u), \gamma_1^{(j)}(u))\\
		=&\det
		\begin{pmatrix}
			A_i(u)\cos\dfrac{x_1(u)}{\lambda} + B_i(u)\sin\dfrac{x_1(u)}{\lambda} &  -B_i(u)\cos\dfrac{x_1(u)}{\lambda} + A_i(u)\sin\dfrac{x_1(u)}{\lambda}
			\vspace{1.5ex}\\
			A_j(u)\cos\dfrac{x_1(u)}{\lambda} + B_j(u)\sin\dfrac{x_1(u)}{\lambda} &  -B_j(u)\cos\dfrac{x_1(u)}{\lambda} + A_j(u)\sin\dfrac{x_1(u)}{\lambda}
		\end{pmatrix}\\
		=&-\det
		\begin{pmatrix}
			A_i(u) & B_i(u) \\
			A_j(u) & B_j(u)
		\end{pmatrix},
	\end{aligned}$$
	where $i,j = 1, 2, ..., 5$ and
	$$\begin{aligned}
		A_1(u) =& a \beta,\\
		B_1(u) =& -\frac{b \beta x_2}{\lambda},\\
		A_2(u) =& a \beta' + b \beta l -\frac{b^2 \beta^2 x_2}{\lambda^2},\\
		B_2(u) =& - \frac{b \beta' x_2}{\lambda} -\frac{2a b \beta^2}{\lambda} -\frac{a \beta l x_2}{\lambda},\\
		A_3(u) =& a \beta'' + 2b \beta' l + b \beta l' + a \beta l^2 - \frac{3a b^2 \beta^3}{\lambda^2} - \frac{3b^2 \beta' \beta x_2}{\lambda^2}	- \frac{3a b \beta^2 l x_2}{\lambda^2},\\
		B_3(u) =& - \frac{3a^2 \beta^2 l}{\lambda} - \frac{3b^2 \beta^2 l}{\lambda} + \frac{b^3 \beta^3 x_2}{\lambda^3} - \frac{b \beta'' x_2}{\lambda}	- \frac{6a b \beta' \beta}{\lambda} - \frac{2a \beta' l x_2}{\lambda} - \frac{a \beta l' x_2}{\lambda} - \frac{b \beta l^2 x_2}{\lambda},\\
		A_4(u) =& a \beta''' + 3b \beta' l' + 3b \beta'' l + b \beta l'' + 3a \beta' l^2 + b \beta l^3 - \frac{3b^2 \beta'^2 x_2}{\lambda^2} + \frac{b^4 \beta^4 x_2}{\lambda^4} + 3a \beta l l' - \frac{6b^3 \beta^3 l}{\lambda^2}\\
		&- \frac{3a^2 \beta^2 l^2 x_2}{\lambda^2} - \frac{4b^2 \beta^2 l^2 x_2}{\lambda^2} - \frac{4b^2 \beta'' \beta x_2}{\lambda^2} - \frac{18a b^2 \beta' \beta^2}{\lambda^2}
		- \frac{12a^2 b \beta^3 l}{\lambda^2} - \frac{4a b \beta^2 l' x_2}{\lambda^2}\\
		&- \frac{14a b \beta' \beta l x_2}{\lambda^2},\\
		B_4(u) =& - \frac{4a^2 \beta^2 l'}{\lambda} - \frac{4b^2 \beta^2 l'}{\lambda} - \frac{b \beta''' x_2}{\lambda} - \frac{6a b \beta'^2}{\lambda} + \frac{4a b^3 \beta^4}{\lambda^3} - \frac{8a b \beta'' \beta}{\lambda} - \frac{3a \beta' l' x_2}{\lambda} - \frac{3a \beta'' l x_2}{\lambda}\\
		&- \frac{a \beta l'' x_2}{\lambda} - \frac{14a^2 \beta' \beta l}{\lambda} - \frac{14b^2 \beta' \beta l}{\lambda} - \frac{a \beta l^3 x_2}{\lambda} - \frac{3b \beta' l^2 x_2}{\lambda} - \frac{14a b \beta^2 l^2}{\lambda} + \frac{6b^3 \beta' \beta^2 x_2}{\lambda^3}\\
		&- \frac{3b \beta l l' x_2}{\lambda} + \frac{6a b^2 \beta^3 l x_2}{\lambda^3},\\
		A_5(u) =& a \beta^{(4)} + 4b\beta' l'' + 6b \beta'' l' + 4b \beta''' l + b \beta l''' + 6a \beta'' l^2 + 3a \beta l'^2 + a \beta l^4 + 4b \beta' l^3 + 12a \beta' l l'\\
		&+ 4a \beta l l'' - \frac{15a^3 \beta^3 l^2}{\lambda^2} + 6b \beta l^2 l' + \frac{5a b^4 \beta^5}{\lambda^4} - \frac{10b^3 \beta^3 l'}{\lambda^2} - \frac{60a b^2 \beta^3 l^2}{\lambda^2} - \frac{10b^2 \beta' \beta'' x_2}{\lambda^2}\\
		&- \frac{5b^2 \beta''' \beta x_2}{\lambda^2} - \frac{45a b^2 \beta'^2 \beta}{\lambda^2} - \frac{30a b^2 \beta'' \beta^2}{\lambda^2} - \frac{20a^2 b \beta^3 l'}{\lambda^2} - \frac{50b^3 \beta' \beta^2 l} {\lambda^2} + \frac{10b^4 \beta' \beta^3 x_2}{\lambda^4}\\
		&- \frac{20a b \beta'^2 l x_2}{\lambda^2} - \frac{5a b \beta^2 l'' x_2}{\lambda^2} - \frac{100a^2 b \beta' \beta^2 l}{\lambda^2} - \frac{15a b \beta^2 l^3 x_2}{\lambda^2} + \frac{10a b^3 \beta^4 l x_2}{\lambda^4} - \frac{20a^2 \beta' \beta l^2 x_2}{\lambda^2}\\
		&- \frac{25b^2 \beta' \beta l^2 x_2}{\lambda^2} - \frac{10a^2 \beta^2 l l' x_2}{\lambda^2} - \frac{15b^2 \beta^2 l l' x_2}{\lambda^2} - \frac{25a b \beta' \beta l' x_2}{\lambda^2} - \frac{25a b\beta''\beta l x_2}{\lambda^2},\\
		B_5(u) =& - \frac{20a^2 \beta'^2 l}{\lambda} - \frac{5a^2 \beta^2 l''}{\lambda} - \frac{20b^2 \beta'^2 l}{\lambda} - \frac{5b^2 \beta^2 l''}{\lambda} + \frac{10b^4 \beta^4 l}{\lambda^3} - \frac{b^5 \beta^5 x_2}{\lambda^5} - \frac{b \beta^{(4)} x_2}{\lambda} - \frac{15a^2 \beta^2 l^3}{\lambda}\\
		&- \frac{15b^2 \beta^2 l^3}{\lambda} - \frac{20a b \beta' \beta''}{\lambda} - \frac{10a b \beta''' \beta}{\lambda} - \frac{4a  \beta' l'' x_2}{\lambda} - \frac{6a \beta'' l' x_2}{\lambda} - \frac{4a \beta''' l x_2}{\lambda} - \frac{a \beta l''' x_2}{\lambda}\\
		&+ \frac{30a^2 b^2 \beta^4 l} {\lambda^3} + \frac{10b^3\beta^3l^2x_2}{\lambda^3} - \frac{25a^2 \beta' \beta l'}{\lambda} - \frac{25a^2 \beta'' \beta l} {\lambda} - \frac{25b^2 \beta' \beta l'}{\lambda} - \frac{25b^2 \beta'' \beta l}{\lambda} - \frac{4a \beta' l^3 x_2}{\lambda}\\
		&- \frac{6b \beta'' l^2 x_2}{\lambda} - \frac{3b \beta l'^2 x_2}{\lambda} - \frac{b \beta l^4 x_2}{\lambda} + \frac{40a b^3 \beta' \beta^3}{\lambda^3} + \frac{15b^3 \beta'^2 \beta x_2}{\lambda^3} + \frac{10b^3\beta''\beta^2x_2}{\lambda^3} - \frac{12b \beta' l l' x_2}{\lambda}\\
		&- \frac{4b \beta l l'' x_2}{\lambda} + \frac{15a^2 b \beta^3 l^2 x_2}{\lambda^3} - \frac{90a b \beta' \beta l^2}{\lambda} - \frac{50a b \beta^2 l l'}{\lambda} - \frac{6a \beta l^2 l' x_2}{\lambda} + \frac{10a b^2 \beta^3 l' x_2}{\lambda^3}\\
		&+ \frac{50a b^2 \beta' \beta^2 l x_2}{\lambda^3}.
	\end{aligned}$$
	If $x_2(u_0) = 0$ and $\beta(u_0)a(u_0) = 0,$ the above equations become 
	$$\begin{aligned}
		A_1(u_0) =& 0,\\
		B_1(u_0) =& 0,\\
		A_2(u_0) =& a \beta' + b \beta l ,\\
		B_2(u_0) =& 0,\\
		A_3(u_0) =& a \beta'' + 2b \beta' l + b \beta l',\\
		B_3(u_0) =& - \frac{3b^2 \beta^2 l}{\lambda},\\
		A_4(u_0) =& a \beta''' + 3b \beta' l' + 3b \beta'' l + b \beta l'' + 3a \beta' l^2 + b \beta l^3 - \frac{6b^3 \beta^3 l}{\lambda^2},\\
		B_4(u_0) =& - \frac{4b^2 \beta^2 l'}{\lambda} - \frac{6a b \beta'^2}{\lambda} - \frac{14b^2 \beta' \beta l}{\lambda},\\
		A_5(u_0) =& a \beta^{(4)} + 4b\beta' l'' + 6b \beta'' l' + 4b \beta''' l + b \beta l''' + 6a \beta'' l^2 + 4b \beta' l^3 + 12a \beta' l l' + 6b \beta l^2 l'\\
		&- \frac{10b^3 \beta^3 l'}{\lambda^2} - \frac{50b^3 \beta' \beta^2 l} {\lambda^2},\\
		B_5(u_0) =& - \frac{20a^2 \beta'^2 l}{\lambda} - \frac{20b^2 \beta'^2 l}{\lambda} - \frac{5b^2 \beta^2 l''}{\lambda} + \frac{10b^4 \beta^4 l}{\lambda^3} - \frac{15b^2 \beta^2 l^3}{\lambda} - \frac{20a b \beta' \beta''}{\lambda} - \frac{25b^2 \beta' \beta l'}{\lambda}\\
		&- \frac{25b^2 \beta'' \beta l}{\lambda}.
	\end{aligned}$$
	
	(1) When $\beta(u_0) = 0, ~ a(u_0) = 0,$ we have
	$$\begin{aligned}
		A_1(u_0) =& 0,\\
		B_1(u_0) =& 0,\\
		A_2(u_0) =& 0 ,\\
		B_2(u_0) =& 0,\\
		A_3(u_0) =& 2b \beta' l,\\
		B_3(u_0) =& 0,\\
		A_4(u_0) =& 3b \beta' l' + 3b \beta'' l,\\
		B_4(u_0) =& 0,\\
		A_5(u_0) =& 4b\beta' l'' + 6b \beta'' l' + 4b \beta''' l + 4b \beta' l^3,\\
		B_5(u_0) =& - \frac{20b^2 \beta'^2 l}{\lambda}.
	\end{aligned}$$
	Moreover, $\gamma_1''(u_0) = \bm{0}, ~ \det(\gamma_1'''(u_0), \gamma_1^{(4)}(u_0)) = 0$ and $\det(\gamma_1'''(u_0), \gamma_1^{(5)}(u_0)) = \dfrac{40b^3 \beta'^3 l^2}{\lambda}.$
	So $\gamma_1$ has
	\begin{itemize}
		\item no $(2,3)$-cusp;
		\item no $(2,5)$-cusp;
		\item no $(3,4)$-cusp;
		\item a $(3,5)$-cusp at $u = u_0$ if and only if $\beta'(u_0)l(u_0) \neq 0.$
	\end{itemize}
	
	(2) When $\beta(u_0) = 0, ~ a(u_0) \neq 0,$ we have
	$$\begin{aligned}
		A_1(u_0) =& 0,\\
		B_1(u_0) =& 0,\\
		A_2(u_0) =& a \beta',\\
		B_2(u_0) =& 0,\\
		A_3(u_0) =& a \beta'' + 2b \beta' l,\\
		B_3(u_0) =& 0,\\
		A_4(u_0) =& a \beta''' + 3b \beta' l' + 3b \beta'' l + 3a \beta' l^2,\\
		B_4(u_0) =& - \frac{6a b \beta'^2}{\lambda},\\
		A_5(u_0) =& a \beta^{(4)} + 4b\beta' l'' + 6b \beta'' l' + 4b \beta''' l + 6a \beta'' l^2 + 4b \beta' l^3 + 12a \beta' l l',\\
		B_5(u_0) =& - \frac{20a^2 \beta'^2 l}{\lambda} - \frac{20b^2 \beta'^2 l}{\lambda} - \frac{20a b \beta' \beta''}{\lambda}.
	\end{aligned}$$
	Moreover, 
	$$\begin{aligned}
		&\det(\gamma_1''(u_0), \gamma_1'''(u_0)) = 0,\\
		&3\det(\gamma_1''(u_0), \gamma_1^{(5)}(u_0)) - 10\det(\gamma_1'''(u_0), \gamma_1^{(4)}(u_0)) = \frac{60\delta a \beta'^3 l}{\lambda},\\
		&\det(\gamma_1'''(u_0), \gamma_1^{(4)}(u_0)) = \frac{6a b \beta'^2}{\lambda}(a \beta'' + 2b \beta' l),\\
		&\det(\gamma_1'''(u_0), \gamma_1^{(5)}(u_0)) = \frac{20\beta'}{\lambda}(a^3 \beta' \beta'' l + 3a b^2 \beta' \beta'' l + a^2b \beta''^2 + 2a^2 b \beta'^2 l^2 + 2b^3 \beta'^2 l^2).
	\end{aligned}$$
	So $\gamma_1$ has
	\begin{itemize}
		\item no $(2,3)$-cusp;
		\item a $(2,5)$-cusp at $u_0$ if and only if $\beta'(u_0)l(u_0) \neq 0$;
		\item no $(3,4)$-cusp;
		\item no $(3,5)$-cusp.
	\end{itemize}
	
	(3) When $\beta(u_0) \neq 0, ~ a(u_0) = 0,$ we have
	$$\begin{aligned}
		A_1(u_0) =& 0,\\
		B_1(u_0) =& 0,\\
		A_2(u_0) =& b \beta l ,\\
		B_2(u_0) =& 0,\\
		A_3(u_0) =& 2b \beta' l + b \beta l',\\
		B_3(u_0) =& - \frac{3b^2 \beta^2 l}{\lambda},\\
		A_4(u_0) =& 3b \beta' l' + 3b \beta'' l + b \beta l'' + b \beta l^3 - \frac{6b^3 \beta^3 l}{\lambda^2},\\
		B_4(u_0) =& - \frac{4b^2 \beta^2 l'}{\lambda} - \frac{14b^2 \beta' \beta l}{\lambda},\\
		A_5(u_0) =& 4b\beta' l'' + 6b \beta'' l' + 4b \beta''' l + b \beta l''' + 4b \beta' l^3 + 6b \beta l^2 l' - \frac{10b^3 \beta^3 l'}{\lambda^2} - \frac{50b^3 \beta' \beta^2 l} {\lambda^2},\\
		B_5(u_0) =& - \frac{20b^2 \beta'^2 l}{\lambda} - \frac{5b^2 \beta^2 l''}{\lambda} + \frac{10b^4 \beta^4 l}{\lambda^3} - \frac{15b^2 \beta^2 l^3}{\lambda} - \frac{25b^2 \beta' \beta l'}{\lambda} - \frac{25b^2 \beta'' \beta l}{\lambda}.
	\end{aligned}$$
	Moreover, $\det(\gamma_1''(u_0), \gamma_1'''(u_0)) = \dfrac{3b^3 \beta^3 l^2}{\lambda}.$
	When $\gamma_1''(u_0) = \bm{0},$
	$$\begin{aligned}
		&\det(\gamma_1'''(u_0), \gamma_1^{(4)}(u_0)) = \frac{4b^3 \beta^3 l'^2}{\lambda},\\
		&\det(\gamma_1'''(u_0), \gamma_1^{(5)}(u_0)) = \frac{5b^3 \beta^2 l'}{\lambda}(\beta l'' + 5\beta' l').
	\end{aligned}$$
So $\gamma_1$ has
	\begin{itemize}
		\item a $(2,3)$-cusp at $u_0$ if and only if $l(u_0) \neq 0$;
		\item no $(2,5)$-cusp;
		\item a $(3,4)$-cusp at $u_0$ if and only if $l(u_0) = 0$ and $l'(u_0) \neq 0$;
		\item no $(3,5)$-cusp.
	\end{itemize}
\end{proof}

\subsection{Singularities of 2-type helicoidal surfaces}

Define the following two maps 
$\varphi_2 : I \times \mathbb{R} \rightarrow I \times \mathbb{R}, ~ (u,v) \mapsto (\tilde u, \tilde v)$ and $\psi_2 : \mathbb{R}_1^3 \rightarrow \mathbb{R}_1^3, ~ (x,y,z) \mapsto (\tilde x, \tilde y, \tilde z)$ by 
$$\varphi_2(u,v) = (u, x_2(u) + \lambda v)$$
and
$$\psi_2(x,y,z) = \left(x \cosh\frac{z}{\lambda} - y \sinh\frac{z}{\lambda}, x \sinh\frac{z}{\lambda} - y \cosh\frac{z}{\lambda}, z \right),$$
respectively. Their Jacobian matrices are 
$$J_{\varphi_2}(u,v) =
\begin{pmatrix}
	1 & 0 \\
	l(u)a(u) & \lambda
\end{pmatrix}$$
and
$$J_{\psi_2}(x,y,z) =
\begin{pmatrix}
	\cosh\dfrac{z}{\lambda} & -\sinh\dfrac{z}{\lambda} & \dfrac{\partial \tilde x}{\partial z}
	\vspace{1.5ex}\\
	\sinh\dfrac{z}{\lambda} & -\cosh\dfrac{z}{\lambda} & \dfrac{\partial \tilde y}{\partial z}
	\vspace{1.5ex}\\
	0 & 0 & 1
\end{pmatrix}.$$
So $\varphi_2$ and $\psi_2$ are both invertible, where the inverse map of $\varphi_2$ is $\varphi_2^{-1}(\tilde u, \tilde v) = \left(\tilde u, \dfrac{\tilde v - x_2(\tilde u)}{\lambda}\right).$

For the 2-type helicoidal surface $\bm{r}_2(u,v) = (x_1(u)\cosh v, x_1(u)\sinh v, x_2(u) + \lambda v),$ we define a surface
$$\begin{aligned}
	&\psi_2 \circ \bm{r}_2 \circ \varphi_2^{-1}(\tilde u, \tilde v)\\
	=& \psi_2 \left(x_1(\tilde u) \cosh \dfrac{\tilde v - x_2(\tilde u)}{\lambda}, x_1(\tilde u) \sinh \dfrac{\tilde v - x_2(\tilde u)}{\lambda}, \tilde v\right)\\
	=& \left(x_1(\tilde u) \cosh \dfrac{x_2(\tilde u)}{\lambda}, x_1(\tilde u) \sinh \dfrac{x_2(\tilde u)}{\lambda}, \tilde v\right)
\end{aligned}$$
and a curve
$$\gamma_2(\tilde u) = \left(x_1(\tilde u) \cosh \dfrac{x_2(\tilde u)}{\lambda}, x_1(\tilde u) \sinh \dfrac{x_2(\tilde u)}{\lambda}\right).$$
Then the 2-type helicoidal surface $\bm{r}_2$ is an $(i,j)$-cuspidal edge at $(u_0, v_0)$ if and only if the surface $\psi_2 \circ \bm{r}_2 \circ \varphi_2^{-1}$ is an $(i,j)$-cuspidal edge at $(u_0, x_2(u_0) + \lambda v_0).$ This is equivalent to the curve $\gamma_2$ being an $(i,j)$-cusp at $\tilde u = u_0.$ Note that $\gamma$ and $\gamma_2$ are diffeomorphic except at $x_1(u_0) = 0.$ We only consider $\gamma_2$ at $x_1(u_0) = 0.$ Since the diffeomorphism $\varphi_2$ maps $u$ to $\tilde u = u,$ we write $\gamma_2(\tilde u)$ as $\gamma_2(u).$ Differentiating $\gamma_2(u)$ with respect to $u,$ we obtain
$$\gamma_2^{(i)}(u) = \left(C_i(u)\cosh\frac{x_2(u)}{\lambda} + D_i(u)\sinh\frac{x_2(u)}{\lambda}, D_i(u)\cosh\frac{x_2(u)}{\lambda} + C_i(u)\sinh\frac{x_2(u)}{\lambda}\right)$$
and
$$\begin{aligned}
	&\det(\gamma_2^{(i)}(u), \gamma_2^{(j)}(u))\\
	=&\det
	\begin{pmatrix}
		C_i(u)\cosh\dfrac{x_2(u)}{\lambda} + D_i(u)\sinh\dfrac{x_2(u)}{\lambda} & D_i(u)\cosh\dfrac{x_2(u)}{\lambda} + C_i(u)\sinh\dfrac{x_2(u)}{\lambda}
		\vspace{1.5ex}\\
		C_j(u)\cosh\dfrac{x_2(u)}{\lambda} + D_j(u)\sinh\dfrac{x_2(u)}{\lambda} & D_j(u)\cosh\dfrac{x_2(u)}{\lambda} + C_j(u)\sinh\dfrac{x_2(u)}{\lambda}
	\end{pmatrix}\\
	=&\det
	\begin{pmatrix}
		C_i(u) & D_i(u) \\
		C_j(u) & D_j(u)
	\end{pmatrix},
\end{aligned}$$
where $i = 1,2,3,4,5$ and
$$\begin{aligned}
	C_1(u) =& b \beta,\\
	D_1(u) =& \frac{a \beta x_1}{\lambda},\\
	C_2(u) =& b \beta' + a \beta l + \frac{a^2 \beta^2 x_1}{\lambda^2},\\
	D_2(u) =& \frac{2 a b \beta^2}{\lambda} + \frac{a \beta' x_1}{\lambda} + \frac{b \beta l x_1}{\lambda},\\
	C_3(u) =& b \beta'' + b \beta l^2 + 2 a \beta' l + a \beta l' + \frac{3 a^2 b \beta^3}{\lambda^2} + \frac{3 a^2 \beta' \beta x_1}{\lambda^2} + \frac{3 a b \beta^2 l x_1}{\lambda^2},\\
	D_3(u) =& \frac{3 a^2 \beta^2 l}{\lambda} + \frac{3 b^2 \beta^2 l}{\lambda} + \frac{a^3 \beta^3 x_1}{\lambda^3} + \frac{a \beta'' x_1}{\lambda} + \frac{a \beta l^2 x_1}{\lambda} + \frac{6 a b \beta' \beta}{\lambda} + \frac{2 b \beta' l x_1}{\lambda} + \frac{b \beta l' x_1}{\lambda},\\
	C_4(u) =& b \beta''' + a \beta l^3 + 3 b \beta' l^2 + 3 a \beta' l' + 3 a \beta'' l + a \beta l'' + \frac{6 a^3 \beta^3 l}{\lambda^2} + \frac{3 a^2 \beta'^2 x_1}{\lambda^2} + \frac{a^4 \beta^4 x_1}{\lambda^4} + 3 b \beta l l'\\
	&+ \frac{18 a^2 b \beta' \beta^2}{\lambda^2} + \frac{12 a b^2 \beta^3 l}{\lambda^2} + \frac{4 a^2 \beta^2 l^2 x_1}{\lambda^2} + \frac{3 b^2 \beta^2 l^2 x_1}{\lambda^2} + \frac{4 a^2 \beta'' \beta x_1}{\lambda^2} + \frac{4 a b \beta^2 l' x_1}{\lambda^2}\\
	&+ \frac{14 a b \beta' \beta l x_1}{\lambda^2},\\
	D_4(u) =& \frac{6 a b \beta'^2}{\lambda} + \frac{4 a^3 b \beta^4}{\lambda^3} + \frac{4 a^2 \beta^2 l'}{\lambda} + \frac{4 b^2 \beta^2 l'}{\lambda} + \frac{a \beta''' x_1}{\lambda} + \frac{3 a \beta' l^2 x_1}{\lambda} + \frac{b \beta l^3 x_1}{\lambda} + \frac{14 a b \beta^2 l^2}{\lambda}\\
	&+ \frac{6 a^3 \beta' \beta^2 x_1}{\lambda^3} + \frac{8 a b \beta'' \beta}{\lambda} + \frac{3 b \beta' l' x_1}{\lambda} + \frac{3 b \beta'' l x_1}{\lambda} + \frac{b \beta l'' x_1}{\lambda} + \frac{14 a^2 \beta' \beta l}{\lambda} + \frac{14 b^2 \beta' \beta l}{\lambda}\\
	&+ \frac{6 a^2 b \beta^3 l x_1}{\lambda^3} + \frac{3 a \beta l l' x_1}{\lambda},\\
	C_5(u) =& b \beta^{(4)} + 4 a \beta' l^3 + 6 b \beta'' l^2 + 3 b \beta l'^2 + b \beta l^4 + 4 a \beta' l'' + 6 a \beta'' l' + 4 a \beta''' l + a \beta l''' + 6 a \beta l^2 l'\\
	&+ \frac{5 a^4 b \beta^5}{\lambda^4} + \frac{10 a^3 \beta^3 l'}{\lambda^2} + 12 b \beta' l l' + 4 b \beta l l'' + \frac{15 b^3 \beta^3 l^2}{\lambda^2} + \frac{45 a^2 b \beta'^2 \beta}{\lambda^2} + \frac{30 a^2 b \beta'' \beta^2}{\lambda^2}\\
	&+ \frac{20 a b^2 \beta^3 l'}{\lambda^2} + \frac{50 a^3 \beta' \beta^2 l}{\lambda^2} + \frac{10 a^4 \beta' \beta^3 x_1}{\lambda^4} + \frac{60 a^2 b \beta^3 l^2}{\lambda^2} + \frac{10 a^2 \beta' \beta'' x_1}{\lambda^2} + \frac{5 a^2 \beta''' \beta x_1}{\lambda^2}\\
	&+ \frac{100 a b^2 \beta' \beta^2 l}{\lambda^2} + \frac{15 a b \beta^2 l^3 x_1}{\lambda^2} + \frac{10 a^3 b \beta^4 l x_1}{\lambda^4} + \frac{25 a^2 \beta' \beta l^2 x_1}{\lambda^2} + \frac{20 b^2 \beta' \beta l^2 x_1}{\lambda^2} + \frac{15 a^2 \beta^2 l l' x_1}{\lambda^2}\\
	&+ \frac{10 b^2 \beta^2 l l' x_1}{\lambda^2} + \frac{20 a b \beta'^2 l x_1}{\lambda^2} + \frac{5 a b \beta^2 l'' x_1}{\lambda^2} + \frac{25 a b \beta' \beta l' x_1}{\lambda^2} + \frac{25 a b \beta'' \beta l x_1}{\lambda^2},\\
	D_5(u) =& \frac{20 a^2 \beta'^2 l}{\lambda} + \frac{5 a^2 \beta^2 l''}{\lambda} + \frac{10 a^4 \beta^4 l}{\lambda^3} + \frac{20 b^2 \beta'^2 l}{\lambda} + \frac{5 b^2 \beta^2 l''}{\lambda} + \frac{a^5 \beta^5 x_1}{\lambda^5} + \frac{a \beta^{(4)} x_1}{\lambda} + \frac{15 a^2 \beta^2 l^3}{\lambda}\\
	&+ \frac{15 b^2 \beta^2 l^3}{\lambda} + \frac{6 a \beta'' l^2 x_1}{\lambda} + \frac{3 a \beta l'^2 x_1}{\lambda} + \frac{a \beta l^4 x_1}{\lambda} + \frac{4 b \beta' l^3 x_1}{\lambda} + \frac{40 a^3 b \beta' \beta^3}{\lambda^3} + \frac{15 a^3 \beta'^2 \beta x_1}{\lambda^3}\\
	&+ \frac{10 a^3 \beta'' \beta^2 x_1}{\lambda^3} + \frac{20 a b \beta' \beta''}{\lambda} + \frac{10 a b \beta''' \beta}{\lambda} + \frac{4 b \beta' l'' x_1}{\lambda} + \frac{6 b \beta'' l' x_1}{\lambda} + \frac{4 b \beta''' l x_1}{\lambda} + \frac{b \beta l''' x_1}{\lambda}\\
	&+ \frac{30 a^2 b^2 \beta^4 l}{\lambda^3} + \frac{10 a^3 \beta^3 l^2 x_1}{\lambda^3} + \frac{25 a^2 \beta' \beta l'}{\lambda} + \frac{25 a^2 \beta'' \beta l}{\lambda} + \frac{25 b^2 \beta' \beta l'}{\lambda} + \frac{25 b^2 \beta'' \beta l}{\lambda}\\
	&+ \frac{10 a^2 b \beta^3 l' x_1}{\lambda^3} + \frac{12 a \beta' l l' x_1}{\lambda} + \frac{4 a \beta l l'' x_1}{\lambda} + \frac{15 a b^2 \beta^3 l^2 x_1}{\lambda^3} + \frac{90 a b \beta' \beta l^2}{\lambda} + \frac{50 a b \beta^2 l l'}{\lambda}\\
	&+ \frac{6 b \beta l^2 l' x_1}{\lambda} + \frac{50 a^2 b \beta' \beta^2 l x_1}{\lambda^3}.
\end{aligned}$$

Similar to the Theorem \ref{th}, we state following conclusions without proof.
\begin{theorem}
	Let $\bm{r}_2$ be the 2-type helicoidal surface and $(u_0, v_0)$ be a singular point of $\bm{r}_2.$ Assume that $x_1(u_0) = 0.$ We obtain the following results.
	\begin{enumerate}
		\item [{\rm (1)}] If $\beta(u_0) = 0, ~ b(u_0) = 0,$ then for any $v \in \mathbb{R},$ the surface $\bm{r}_2$ is a $(3,5)$-cuspidal edge at $(u_0, v)$ if and only if $\beta'(u_0)l(u_0) \neq 0.$ In this case, $\bm{r}_2$ does not have $(2,3)$-cuspidal edges, $(2,5)$-cuspidal edges or $(3,4)$-cuspidal edges.
		\item [{\rm (2)}] If $\beta(u_0) = 0, ~ b(u_0) \neq 0,$ then for any $v \in \mathbb{R},$ the surface $\bm{r}_2$ is a $(2,5)$-cuspidal edge at $(u_0, v)$ if and only if $\beta'(u_0)l(u_0) \neq 0.$ In this case, $\bm{r}_2$ does not have $(2,3)$-cuspidal edges, $(3,4)$-cuspidal edges or $(3,5)$-cuspidal edges.
		\item [{\rm (3)}] If $\beta(u_0) \neq 0, ~ b(u_0) = 0,$ then for any $v \in \mathbb{R},$ the surface $\bm{r}_2$ is a $(2,3)$-cuspidal edge at $(u_0, v)$ if and only if $l(u_0) \neq 0.$ $\bm{r}_2$ is a $(3,4)$-cuspidal edge at $(u_0, v)$ if and only if $l(u_0) = 0$ and $l'(u_0) \neq 0.$ In this case, $\bm{r}_2$ does not have $(2,5)$-cuspidal edges or $(3,5)$-cuspidal edges.
	\end{enumerate}
\end{theorem}

\section{Examples}
\begin{example}
	Let $(\gamma, \bm\nu) : \mathbb{R} \rightarrow \mathbb{R}_1^2 \times \Delta$ be 
	$$\gamma(u) = (u \cosh u - \sinh u, u \sinh u - \cosh u + 1), ~ \bm\nu(u) = (\cosh u, \sinh u).$$
	$(\gamma, \bm\nu)$ is a spacelike Legendre curve with the curvature $l(u) = 1$ and $\beta(u) = u.$ Take $\lambda = 1,$ then the 1-type helicoidal surface is 
	$$\bm{r}_1(u,v) = (u \cosh u - \sinh u + v, (u \sinh u - \cosh u + 1)\sin v, (u \sinh u - \cosh u + 1)\cos v).$$
	Moreover, 
	$$\begin{aligned}
		\frac{\partial \bm{r}_1}{\partial u}(u,v) &= (u \sinh u, u \cosh u \sin v, u \cosh u \cos v),\\
		\frac{\partial \bm{r}_1}{\partial v}(u,v) &= (1, (u \sinh u - \cosh u + 1)\cos v, -(u \sinh u - \cosh u + 1)\sin v),\\
		\frac{\partial \bm{r}_1}{\partial u}(u,v) \wedge \frac{\partial \bm{r}_1}{\partial v}(u,v) &= u\big((u \sinh u - \cosh u + 1)\cosh u,\\
		&~~~~~~(u \sinh u - \cosh u + 1)\sinh u \sin v + \cosh u \cos v,\\
		&~~~~~~(u \sinh u - \cosh u + 1)\sinh u \cos v - \cosh u \sin v \big).
	\end{aligned}$$
	The singularities of $\bm{r}_1$ are $(0, v)$ for any $v \in \mathbb{R}.$ When $u = 0,$ we have
	$$x_2(0) = 0, ~ \beta(0) = 0, ~ a(0) = 1, ~ \beta'(0)l(0) = 1.$$
	So $\bm{r}_1$ is a $(2,5)$-cuspidal edge at $(0,v).$ $\bm{r}_1$ and its singular locus are shown in Figure \ref{ex1}.
	
	\begin{figure}[H]
		\centering
		\includegraphics[width = 9cm, height = 10cm]{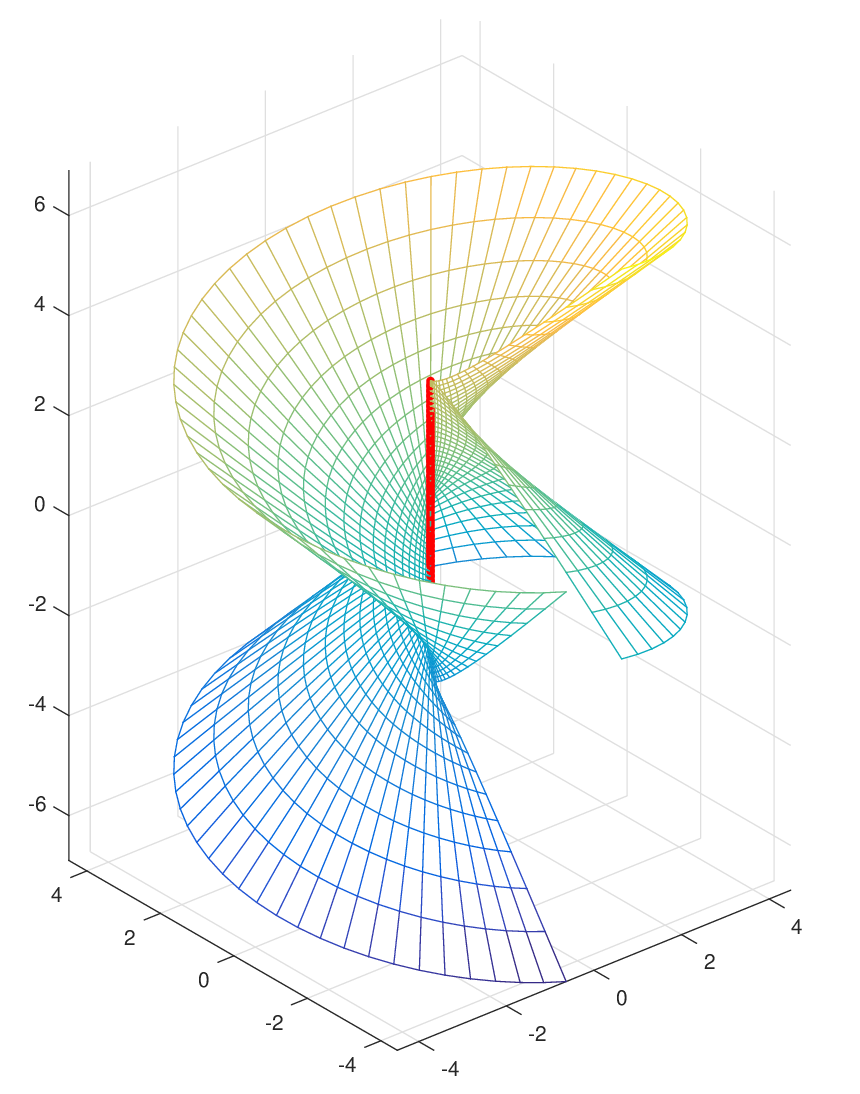}
		\caption{The 1-type helicoidal surface (mesh) and its singular locus (red curve).}
		\label{ex1}
	\end{figure}
\end{example}

\begin{example}
	Let $(\gamma, \bm\nu) : (-1,1) \rightarrow \mathbb{R}_1^2 \times \Delta$ be
	$$\gamma(u) = \left( \frac{u^2}{2}, \frac{u^3}{3} \right), ~ \bm\nu(u) = \frac{(u,1)}{\sqrt{1-u^2}}.$$
	$(\gamma, \bm\nu)$ is a timelike Legendre curve with the curvature $l(u) = \dfrac{1}{1 - u^2}$ and $\beta(u) = u\sqrt{1 - u^2}.$ If we take $\lambda = 1,$ then the 2-type helicoidal surface is
	$$\bm{r}_2(u,v) = \left(\frac{u^2}{2}\cosh v, \frac{u^2}{2}\sinh v, \frac{u^3}{3} + v\right).$$
	Moreover,
	$$\begin{aligned}
		\frac{\partial \bm{r}_2}{\partial u}(u,v) &= (u \cosh v, u \sinh v, u^2),\\
		\frac{\partial \bm{r}_2}{\partial v}(u,v) &= \left(\frac{u^2}{2}\sinh v, \frac{u^2}{2}\cosh v, 1\right),\\
		\frac{\partial \bm{r}_2}{\partial u}(u,v) \wedge \frac{\partial \bm{r}_2}{\partial v}(u,v) &= u\left(\frac{u^3}{2} \cosh v - \sinh v, \frac{u^3}{2} \sinh v - \cosh v, \frac{u^2}{2}\right).
	\end{aligned}$$
	The singularities of $\bm{r}_2$ are $(0, v)$ for any $v \in \mathbb{R}.$ When $u = 0,$ we have
	$$x_1(0) = 0, ~ \beta(0) = 0, ~ b(0) = 1, ~ \beta'(0)l(0) = 1.$$
	So $\bm{r}_2$ is a $(2,5)$-cuspidal edge at $(0,v).$ $\bm{r}_2$ and its singular locus are shown in Figure \ref{ex2}.
	
	\begin{figure}[H]
		\centering
		\includegraphics[width = 12cm, height = 7cm]{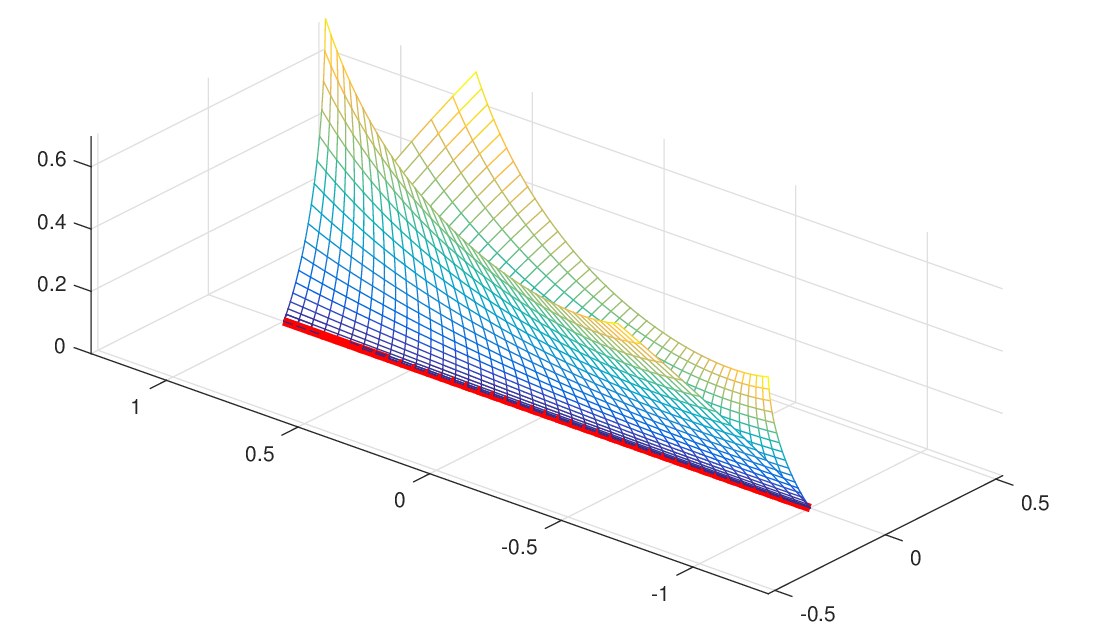}
		\caption{The 2-type helicoidal surface (mesh) and its singular locus (red curve).}
		\label{ex2}
	\end{figure}
\end{example}

\bigskip

\noindent{\bf Acknowledgements.}
The first author is funded by the Science Research Project of Hebei Education Department (Grant No. QN2026104).
The second author is partially supported by the Yili Normal University Returning Doctoral Research Start-up Project (Grant No. 2025GFX001) and National Nature Science Foundation of China (Grant No. 12471021).


\end{document}